\documentclass[preprint,12pt]{elsarticle}

\usepackage{amsmath}
\usepackage{amsfonts}
\usepackage{amssymb}
\usepackage{amsthm}
\usepackage{booktabs}
\usepackage{ragged2e}
\usepackage[utf8]{inputenc}
\usepackage[english]{babel}
\usepackage{graphicx}
 \usepackage{parskip}
\usepackage{textpos}
\usepackage[utf8]{inputenc}
\usepackage[T1]{fontenc}
\usepackage{helvet}
\usepackage{amsfonts}
\usepackage{amsmath}
\usepackage{amssymb}
\usepackage{tikz, graphicx}
\usepackage{hyperref}
\biboptions{numbers,sort&compress}
\usepackage{bigints}
\usepackage{xcolor}
\usepackage{algorithm}
\usepackage{algpseudocode}
\usepackage{geometry}

\newtheorem{theorem}{Theorem}[section]

\newtheorem{definition}[theorem]{Definition}
\newtheorem{corollary}[theorem]{Corollary}

\newtheorem{lemma}[theorem]{Lemma}

\newtheorem{remark}[theorem]{Remark}

\usepackage{float}
\usepackage{graphicx} 
\numberwithin{equation}{section}

\begin{document}
\begin{frontmatter}



 \title{\textbf{A Class of Function-Preserving Neural Network Operators with Applications to Modulated Signal Denoising}}



\author{\textbf{Sachin Saini}}
\ead{sachin\_s@ma.iitr.ac.in; sbisaini003@gmail.com}


\begin{abstract}
In this paper, we introduce a class of function-preserving neural network operators designed to preserve a prescribed strictly positive function.
We examine the approximation behaviour of the proposed operators on the space of continuous functions and show that they are positive, linear, well-defined, and bounded. Specifically, we demonstrate uniform convergence on compact intervals and get quantitative estimates of the approximation error in terms of the modulus of continuity. To support the theoretical analysis, we present several numerical experiments focusing on the preservation of exponential and logarithmic functions. These examples show how well the proposed operators perform and compare them with existing neural network operators. Also, an application to denoising modulated signals is provided, demonstrating the practical relevance of the proposed approach.
\end{abstract}

\begin{keyword}
Sigmoidal function\sep Neural network operators\sep Function preservation\sep Uniform convergence\sep Signal denoising.

\MSC[2020] 41A25\sep 41A30\sep 41A35\sep 41A36\sep  68T07.
\end{keyword}

\end{frontmatter}

\section{Introduction}

The study of \emph{neural network operators} (NNOs) has emerged as a practical and effective approach to function approximation by combining concepts from neural network models with traditional approximation theory. These operators offer explicit representations, enabling the approximation of functions using structures inspired by brain architecture, rather than relying solely on data-driven approaches.

Two significant families of NNOs, the \emph{bell-shaped} and \emph{squashing-type} operators, were presented by Cardaliaguet and Euvrard~\cite{cardaliaguet1992approximation}. This contribution demonstrated a rigorous and explicit relationship between classical approximation theory and neural network models. Building on this framework, Anastassiou~\cite{anastassiou1997rate,anastassiou2000rate,anastassiou2011multivariatesigmoidal,anastassiou2013frac,anastassiou2017approximations}
developed an extensive theory for NNOs, including quantitative rates of convergence, fractional extensions, and multivariate generalizations for a wide class of activation functions. In recent years, the theory has witnessed significant advancements, including density results, quantitative estimates, and nonlinear extensions of  NNOs, which have substantially broadened the scope of the approximation framework~\cite{costarelli2022density,costarelli2022quantitative,costarelli2024durrmeyer,berisha2025riemann,saini2026approximation,kadak2023neural,lin2026construction,chen2015approximation,qian2022neural,majethiya2025modular,baxhaku2025riemann}. These developments have enabled the study of a wider class of activation functions and facilitated more general operator-based constructions.

A significant step forward in the study of NNOs was provided by Costarelli and Spigler~\cite{costarelli2013approxi}. They introduced a family of positive linear operators defined on $C(\mathcal{D})$, where $\mathcal{D}=[c,d]\subset\mathbb{R}$. These operators are generated by sigmoidal activation functions and are defined as
\begin{equation}\label{clasicalLNNO1}
F_m(f;v)
=
\frac{\sum\limits_{j=\lceil mc\rceil}^{\lfloor md\rfloor}
\varrho_\rho(mv-j)\,f\!\left(\tfrac{j}{m}\right)}
{\sum\limits_{j=\lceil mc\rceil}^{\lfloor md\rfloor}
\varrho_\rho(mv-j)},
\quad v\in\mathcal{D}.
\end{equation}
These operators are essential tools in neural network approximation theory, since they are positive, linear, and reproduce constant functions. A Kantorovich-type modification was later proposed in~\cite{costarelli2014convergence}, where pointwise evaluations are replaced by local averages,
thereby extending approximation results to $L^p(\mathcal{D})$, $1\le p<\infty$.

Parallel to these developments, classical approximation theory has long emphasized the importance of \emph{function-preserving operators}. In particular, Bernstein-type constructions and their variants (see, e.g., \cite{lorentz1986bernstein,cardenas2006shape,bede2009approximation,aldaz2009shape,acu2024new,su2022shape,ayman2025shape}) are designed to preserve specific functional features, such as exponential or logarithmic behaviour. Such preservation often leads to significantly improved approximation when the target function exhibits non-polynomial characteristics.

More recently, sampling-type NNOs based on exponential grids have been introduced in~\cite{bajpeyi2021exponential,bajpeyi2023order}. These operators combine sigmoidal activation functions with exponentially spaced nodes and are given by
These operators incorporate a logarithmic-type scaling via exponential sampling and are particularly well-suited to functions defined on positive domains.

Despite their flexibility and strong approximation properties, the existing NNOs~\eqref{clasicalLNNO1} are essentially based on additive or sampling-type constructions and, in general, do not preserve nonlinear structures such as logarithmic, exponential, or multiplicative behaviour. In particular, even in the exponential setting, such structures are only approximated asymptotically and not reproduced exactly. For instance, in general, one has
\[F_m(\ln f;v) \neq \ln\big(F_m(f;v)\big),\]
indicating that such operators are not compatible with logarithmic transformations.

This limitation becomes significant in many applications, including signal processing, financial modelling, and multiplicative systems, where functions often exhibit inherent structural patterns such as exponential growth, logarithmic scaling, or power-law behaviour. In these cases, existing NNOs may struggle to adapt to the problem's intrinsic structure, leading to less effective approximations.

Inspired by these ideas, we provide a novel class of \emph{function-preserving neural network operators}. In order to better depict a particular structural aspect in the resulting approximation, the basic concept is to construct the operators using a prescribed strictly positive function $g$.

More specifically, the operators $T_m^{(g)}$ are defined to satisfy the reproduction property
\[T_m^{(g)}(g;v)=g(v),\] that establishes them apart from other NNOs.

Notable special cases are obtained by taking
\[g(v)=\ln(1+\mu+v) \ \text{and} \ g(v)=e^{\lambda v}, \]
which yield logarithmic and exponential operators, respectively. These constructions preserve the underlying structure functions exactly and are particularly suitable for approximating functions with logarithmic or exponential behaviour.

By including structural preservation into neural network-based approximation, the current framework thereby generalizes existing neural network operator theory and achieves greater flexibility and enhanced approximation accuracy.

\section{Preliminaries}\label{sec:preliminaries}
Let \(\mathcal{D} = [c,d] \subset \mathbb{R}\).  
We denote by \(C(\mathcal{D})\) the set of all continuous real-valued functions defined on \(\mathcal{D}\), equipped with the supremum norm
\[\|f\|_\infty = \sup_{v \in \mathcal{D}} |f(v)|.\]
The term \emph{sigmoidal function} refers to a measurable function $\rho: \mathbb{R} \to \mathbb{R}$ that satisfies
\[\lim_{v \to -\infty} \rho(v) = 0, \quad \lim_{v \to +\infty} \rho(v) = 1.\]
Throughout this work, we assume that \(\rho\) is a non-decreasing sigmoidal function such that it adheres to the conditions listed below:

\begin{enumerate}
    \item[(P1)] $\rho(v)-\tfrac12$ is an odd function, i.e.,
    \[\rho(-v)=1-\rho(v), \quad v\in\mathbb{R};\]
    
    \item[(P2)] $\rho\in C^2(\mathbb{R})$, $\rho$ is increasing on $\mathbb{R}$, and concave for $v\ge0$;
    
    \item[(P3)] $\rho$ has sufficient decay at $-\infty$, namely
    \[
    \rho(v)=\mathcal{O}(|v|^{-1-\beta}) \quad \text{as } v\to -\infty,
    \]
    for some $\beta>1$;
    
    \item[(P4)] $\rho'\in L^\infty(\mathbb{R})$.
\end{enumerate}

\noindent
Under assumptions (P1)-(P3), it follows that
\[ \rho(v)=1-\mathcal{O}(|v|^{-1-\beta}) \quad \text{as } v\to +\infty.\]
The associated activation function is defined by
\[ \varrho_\rho(v) := \frac12\big[\rho(v+1)-\rho(v-1)\big], \quad v\in\mathbb{R}.\]
For more details, one can see \cite{costarelli2024asymptotic,xue2025direct}.

\begin{definition}
The discrete absolute moment of order $\alpha$ for $\alpha\ge0$ is defined as
\[ \mathcal{M}_{\alpha}(\varrho_\rho) := \sup_{v\in\mathbb{R}} \sum_{j\in\mathbb{Z}} |\varrho_\rho(v-j)|\,|v-j|^{\alpha}.\]
\end{definition}

\noindent
Under assumption (P3), it holds that $\mathcal{M}_\alpha(\varrho_\rho)<\infty$ for all $0\le \alpha<\beta$. In particular, since $\beta>1$, we have $\mathcal{M}_1(\varrho_\rho)<\infty$.
For details, one can see \cite{costarelli2024asymptotic}.

\noindent
The following lemma collects the fundamental properties of $\varrho_\rho$ (see, e.g., \cite{costarelli2013approxi,costarelli2014convergence}).

\begin{lemma}\label{lemma-psi-properties}
Let $\varrho_\rho$ be defined as above. Then:
\begin{enumerate}
\item[(i)] $\varrho_\rho(v)\ge0$ for all $v\in\mathbb{R}$ and
\[ \lim_{v\to\pm\infty}\varrho_\rho(v)=0;\]
\item[(ii)] $\varrho_\rho$ is an even function;
    
\item[(iii)]  \(
\sum\limits_{j\in\mathbb{Z}}\varrho_\rho(v-j)=1, \ v\in\mathbb{R},
\) and \( \|\varrho_\rho\|_{L^1(\mathbb{R})}=1; \)
    
\item[(iv)] (Localization)  
For every $\delta>0$,
\[ \lim_{m\to\infty} \sum_{|mv-j|>m\delta}\varrho_\rho(mv-j)=0, \]  uniformly in $v\in\mathbb{R}$;
    
\item[(v)]   
There exists a constant \(   c_0 := \min\limits_{|x|\le 1}\varrho_\rho(x) > 0
\) such that, for all $v\in[c,d]$ and $n\in\mathbb{N}$,
\[\sum_{j=\lceil mc\rceil}^{\lfloor md\rfloor} \varrho_\rho(mv-j)\ge c_0.\]
Indeed, for each $v\in[c,d]$, there exists at least one integer $j$ such that $|mv-j|\le1$;
    
\item[(vi)] $\varrho_\rho(v)=\mathcal{O}(|v|^{-1-\beta})$ as $v\to\pm\infty$;
    
\item[(vii)] The series $\sum\limits_{j\in\mathbb{Z}}\varrho_\rho(v-j)$ converges uniformly on compact sets;
    
\item[(viii)] $\varrho_\rho$ is Lipschitz continuous, and
\[ |\varrho_\rho'(v)| \le \tfrac12\big(|\rho'(v+1)|+|\rho'(v-1)|\big)
    \le \|\rho'\|_\infty, \]
so that a Lipschitz constant is given by $L\le \|\rho'\|_\infty$.
\end{enumerate}
\end{lemma}

\section{Function-Preserving Neural Network Operators}
In this section, we develop a general class of NNOs designed to preserve a prescribed functional structure.

\noindent
Let $g:\mathcal{D}\to\mathbb{R}$ be a continuous function satisfying
\[ g(v)>0, \quad v\in\mathcal{D}. \]
Since $g$ is continuous and strictly positive on the compact interval $\mathcal{D}$, there exist constants $0<m_g \le M_g < \infty$ such that
\[ 0 < m_g \le g(v) \le M_g, \quad v\in\mathcal{D}. \]

\begin{definition}
For \( f \in C(\mathcal{D}) \) and \( m \in \mathbb{N} \), the \emph{\( g \)-preserving (function-preserving) neural network operator} is defined by
\begin{equation}\label{def-gNNO}
T_m^{(g)}(f;v) = g(v)\,
\frac{\displaystyle\sum\limits_{j=\lceil mc\rceil}^{\lfloor md\rfloor}
\varrho_\rho\big(m a_m(v)-j\big)\,
\dfrac{f\!\left(\tfrac{j}{m}\right)}{g\!\left(\tfrac{j}{m}\right)}}
{\displaystyle\sum\limits_{j=\lceil mc\rceil}^{\lfloor md\rfloor}
\varrho_\rho\big(m a_m(v)-j\big)}, \quad v\in\mathcal{D},
\end{equation}
\end{definition}

\noindent
where, $\{a_m\}_{m\in\mathbb{N}}$ be a sequence of continuous mappings $a_m:\mathcal{D}\to\mathcal{D}$ satisfying
\[ a_m(\mathcal{D}) \subseteq \mathcal{D}, \quad \lim_{m\to\infty} \|a_m - \mathrm{Id}\|_\infty = 0, \quad m\|a_m - \mathrm{Id}\|_\infty \le C_a,
\]
$\mathrm{Id}(v)=v$ and $C_a>0$ is independent of $n$.

\noindent
For convenience, define the $g$-transformation
\[ f_g(v):=\frac{f(v)}{g(v)}, \quad v\in\mathcal{D}. \] Since $g(v)\ge m_g>0$, it follows that $f_g\in C(\mathcal{D})$.
\noindent
Then, the operator \eqref{def-gNNO} can be equivalently written as
\begin{equation}\label{repr-gNNO}
T_m^{(g)}(f;v) = g(v)\,F_m^{(a)}(f_g;v),
\end{equation}
where
\[ F_m^{(a)}(h;v) = \frac{\displaystyle\sum\limits_{j=\lceil mc\rceil}^{\lfloor md\rfloor}
\varrho_\rho\big(m a_m(v)-j\big)\,h\!\left(\tfrac{j}{m}\right)}
{\displaystyle\sum\limits_{j=\lceil mc\rceil}^{\lfloor md\rfloor}
\varrho_\rho\big(m a_m(v)-j\big)},
\quad h\in C(\mathcal{D}). \]
The operator $F_m^{(a)}$ is a positive linear operator with respect to $h$ and satisfies the normalization property
\[F_m^{(a)}(1;v)=1, \quad v\in\mathcal{D}.\]
Now, we present two important special cases corresponding to logarithmic and exponential structure functions.

\medskip
\noindent\textbf{(i) Logarithmic Preservation.}
Let $\mu>0$ and assume that
\[ 1+\mu+v > 1, \quad v\in\mathcal{D},\]
so that $\ln_\mu(v):=\ln(1+\mu+v)$ is well-defined and strictly positive on $\mathcal{D}$. Then $\ln_\mu$ satisfies
\[ 0 < m_{\log} \le \ln_\mu(v) \le M_{\log}, \quad v\in\mathcal{D}. \]
The corresponding operator is given by
\begin{align*}
T_m^{(\log)}(f;v)
& =\ln_\mu(v)\,F_m^{(a)}\!\left(\frac{f}{\ln_\mu};v\right)\\
&= \ln_\mu(v)\,
\frac{\displaystyle\sum\limits_{j=\lceil mc\rceil}^{\lfloor md\rfloor}
\varrho_\rho\big(m a_m(v)-j\big)\,
\dfrac{f\!\left(\tfrac{j}{m}\right)}{\ln_\mu\!\left(\tfrac{j}{m}\right)}}
{\displaystyle\sum\limits_{j=\lceil mc\rceil}^{\lfloor md\rfloor}
\varrho_\rho\big(m a_m(v)-j\big)}.
\end{align*}
Moreover, $T_m^{(\log)}$ reproduces $\ln_\mu$, i.e.,
\[ T_m^{(\log)}(\ln_\mu;v)=\ln_\mu(v), \quad v\in\mathcal{D}. \]

\noindent\textbf{(ii) Exponential Preservation.}
Let $\lambda>0$ and define $g(v)=e^{\lambda v}$. Then $g$ is strictly positive and continuous on $\mathcal{D}$, and satisfies
\[ 0 < m_{\exp} \le e^{\lambda v} \le M_{\exp}, \quad v\in\mathcal{D}.\]
The corresponding operator is
\begin{align*}
T_m^{(\exp)}(f;v)
&= e^{\lambda v}\,F_m^{(a)}\!\left(f(\cdot)e^{-\lambda(\cdot)};v\right)\\
&= e^{\lambda v}\,
\frac{\displaystyle\sum\limits_{j=\lceil mc\rceil}^{\lfloor md\rfloor}
\varrho_\rho\big(m a_m(v)-j\big)\,
\dfrac{f\!\left(\tfrac{j}{m}\right)}{e^{\lambda j/m}}}
{\displaystyle\sum\limits_{j=\lceil mc\rceil}^{\lfloor md\rfloor}
\varrho_\rho\big(m a_m(v)-j\big)}.
\end{align*}
Moreover, $T_m^{(\exp)}$ reproduces $e^{\lambda v}$, i.e.,
\[ T_m^{(\exp)}(e^{\lambda v};v)=e^{\lambda v}, \quad v\in\mathcal{D}. \]
We first show that the proposed operators are bounded and well-defined.

\begin{theorem}
Let $g\in C(\mathcal{D})$ satisfy $g(v)>0$ for all $v\in\mathcal{D}$, and let $f\in C(\mathcal{D})$. Then the operator $T_m^{(g)}(f;v)$ is well-defined for all $v\in\mathcal{D}$. Moreover, 
\[ |T_m^{(g)}(f;v)| \le \frac{M_g}{m_g} \|f\|_\infty, \quad v\in\mathcal{D}.\]
\end{theorem}

\begin{proof}
Since $g\in C(\mathcal{D})$ and $g(v)>0$ on the compact interval $\mathcal{D}$, there exist constants
\[ m_g := \min_{v\in\mathcal{D}} g(v) > 0, \quad M_g := \max_{v\in\mathcal{D}} g(v) < \infty.\]

\noindent
\textbf{(i) Well-definedness.}
For each $v\in\mathcal{D}$, we have $a_m(v)\in\mathcal{D}$, and hence $m a_m(v)\in [mc, md]$. Therefore, by property (v) of Lemma~\ref{lemma-psi-properties}, $\exists \ c_0>0$ satisfying
\[ \sum_{j=\lceil mc\rceil}^{\lfloor md\rfloor} \varrho_\rho\big(m a_m(v)-j\big)
\ge c_0, \quad v\in\mathcal{D},\; m\in\mathbb{N}.\]
Thus, the denominator in the definition of $T_m^{(g)}$ is strictly positive and uniformly bounded away from zero.

\noindent
Moreover, since $\tfrac{j}{m}\in\mathcal{D}$ and $g(j/m)\ge m_g>0$ for all $j$, the quantity \[ \frac{f(j/m)}{g(j/m)} \] is well-defined and bounded. Since the sums involved are finite, it follows that $T_m^{(g)}(f;v)$ is well-defined for all $v\in\mathcal{D}$.

\medskip

\noindent\textbf{(ii) Boundedness.}
Define the normalized weights
\[ w_{m,j}(v) := \frac{\varrho_\rho(m a_m(v)-j)}
{\sum\limits_{k=\lceil mc\rceil}^{\lfloor md\rfloor} \varrho_\rho(m a_m(v)-k)}. \]
Then $w_{m,j}(v)\ge 0$ and
\[ \sum_{j=\lceil mc\rceil}^{\lfloor md\rfloor} w_{m,j}(v)=1. \]
Using the representation \eqref{repr-gNNO}, we obtain
\[ |T_m^{(g)}(f;v)| \le g(v)\sum_{j=\lceil mc\rceil}^{\lfloor md\rfloor}
w_{m,j}(v) \left|\frac{f(j/m)}{g(j/m)}\right|. \]
Since $g(j/m)\ge m_g$, it follows that
\[ \left|\frac{f(j/m)}{g(j/m)}\right| \le \frac{\|f\|_\infty}{m_g}. \]
Hence,
\begin{align*}
|T_m^{(g)}(f;v)| &\le g(v)\cdot \frac{\|f\|_\infty}{m_g}
\sum_{j=\lceil mc\rceil}^{\lfloor md\rfloor} w_{m,j}(v) \\
& = \frac{g(v)}{m_g}\,\|f\|_\infty \\
&\le \frac{M_g}{m_g}\,\|f\|_\infty.
\end{align*}
This proves the desired estimate.
\end{proof}

\begin{corollary}
Let $\mu>0$ and assume that \( 1+\mu+v > 1, \ v\in\mathcal{D},
\) so that $\ln_\mu(v)=\ln(1+\mu+v)$ is strictly positive on $\mathcal{D}$. Then the operator $T_m^{(\log)}$ is well-defined and satisfies
\[ |T_m^{(\log)}(f;v)| \le \frac{M_{\log}}{m_{\log}} \|f\|_\infty, \quad v\in\mathcal{D}.\]
\end{corollary}

\begin{corollary}
Let $\lambda>0$ and define $g(v)=e^{\lambda v}$ for $v\in\mathcal{D}$. Then the operator $T_m^{(\exp)}$ is well-defined and satisfies
\[ |T_m^{(\exp)}(f;v)| \le \frac{M_{\exp}}{m_{\exp}} \|f\|_\infty, \quad v\in\mathcal{D}.\]
\end{corollary}

\noindent
Now, we establish the positivity, linearity, and preservation properties of the proposed NNOs.

\begin{theorem}
Let $g \in C(\mathcal{D})$ satisfy $g(v)>0$ for all $v \in \mathcal{D}$. Then the operator
\[ T_m^{(g)} : C(\mathcal{D}) \to C(\mathcal{D}) \] is positive, linear, and reproduces the function $g$, i.e.,
\[ T_m^{(g)}(g;v) = g(v), \quad v \in \mathcal{D}. \]
\end{theorem}

\begin{proof}
Fix $v \in \mathcal{D}$. By property (v) of Lemma~\ref{lemma-psi-properties},  $\exists \ c_0>0$ satisfying \[
\sum_{j=\lceil mc\rceil}^{\lfloor md\rfloor} \varrho_\rho(m a_m(v)-j) \ge c_0, \quad v\in\mathcal{D},\; m\in\mathbb{N}. \]
Therefore, the denominator in \eqref{def-gNNO} is strictly positive and uniformly bounded away from zero.

\noindent
Thus, the operator admits the representation
\begin{equation}\label{repr-weight-clean-final}
T_m^{(g)}(f;v) = g(v)\sum_{j=\lceil mc\rceil}^{\lfloor md\rfloor}
w_{m,j}(v)\, \frac{f\!\left(\tfrac{j}{m}\right)}{g\!\left(\tfrac{j}{m}\right)},
\end{equation}
where
\[ w_{m,j}(v) := \frac{\varrho_\rho(m a_m(v)-j)}
{\sum\limits_{k=\lceil mc\rceil}^{\lfloor md\rfloor} \varrho_\rho(m a_m(v)-k)}. \]
Then $w_{m,j}(v)\ge 0$ and
\[ \sum_{j=\lceil mc\rceil}^{\lfloor md\rfloor} w_{m,j}(v)=1. \] Since $a_m$, $\varrho_\rho$, and $g$ are continuous, and the denominator is bounded away from zero, it follows that each $w_{m,j}(v)$ is continuous in $v$. Hence, as a finite sum of continuous functions,
\[ T_m^{(g)}(f;\cdot)\in C(\mathcal{D}), \quad \forall f\in C(\mathcal{D}),\]
so that $T_m^{(g)}:C(\mathcal{D}) \to C(\mathcal{D})$.

\medskip
\noindent
\textbf{(i) Positivity.}
Let $f \in C(\mathcal{D})$ with $f \ge 0$ on $\mathcal{D}$. Since $g(j/m)>0$ for all $j$, we have
\[ \frac{f\!\left(\tfrac{j}{m}\right)}{g\!\left(\tfrac{j}{m}\right)} \ge 0.\]
Using \eqref{repr-weight-clean-final} and $w_{m,j}(v)\ge 0$, it follows that
\[ T_m^{(g)}(f;v)\ge 0, \quad v\in\mathcal{D}.\]
\textbf{(ii) Linearity.}
Let $f,h \in C(\mathcal{D})$ and $\alpha,\beta \in \mathbb{R}$. Since the weights $w_{m,j}(v)$ are independent of the function, we obtain
\begin{align*}
T_m^{(g)}(\alpha f + \beta h; v)
& =g(v)\sum_{j=\lceil mc\rceil}^{\lfloor md\rfloor} w_{m,j}(v) \frac{\alpha f\!\left(\tfrac{j}{m}\right) + \beta h\!\left(\tfrac{j}{m}\right)}
{g\!\left(\tfrac{j}{m}\right)} \\
& = \alpha T_m^{(g)}(f;v) + \beta T_m^{(g)}(h;v).
\end{align*}
Thus, $T_m^{(g)}$ is linear.

\medskip
\noindent
\textbf{(iii) Reproduction of $g$.}
Taking $f=g$, we obtain
\[ \frac{g\!\left(\tfrac{j}{m}\right)}{g\!\left(\tfrac{j}{m}\right)} = 1, \]
and hence
\[ T_m^{(g)}(g;v) = g(v)\sum_{j=\lceil mc\rceil}^{\lfloor md\rfloor} w_{m,j}(v)
= g(v), \]
since $\sum_{j=\lceil mc\rceil}^{\lfloor md\rfloor} w_{m,j}(v)=1$.
\end{proof}

\begin{corollary}
Let $\mu>0$ and assume that \( 1+\mu+v > 1, \ v\in\mathcal{D},
\) so that $\ln_\mu(v)=\ln(1+\mu+v)$ is strictly positive on $\mathcal{D}$. Then the operator $T_m^{(\log)}$ is positive, linear, and reproduces $\ln_\mu$, i.e.,
\[ T_m^{(\log)}(\ln_\mu;v)=\ln_\mu(v), \quad v\in\mathcal{D}.\]
\end{corollary}

\begin{corollary}
Let $\lambda>0$ and define $g(v)=e^{\lambda v}$ for $v\in\mathcal{D}$. Then the operator $T_m^{(\exp)}$ is positive, linear, and reproduces $e^{\lambda v}$, i.e.,
\[ T_m^{(\exp)}(e^{\lambda v};v)=e^{\lambda v}, \quad v\in\mathcal{D}.\]
\end{corollary}

\begin{remark}
In general, the operator $T_m^{(g)}$ does not reproduce constant functions. Indeed,
\[ T_m^{(g)}(1;v) = g(v)\,F_m^{(a)}\!\left(\frac{1}{g};v\right), \]
which is not identically equal to $1$ unless $g$ is constant. This reflects the function-preserving nature of the operator, which is designed to reproduce $g$ rather than constant functions.
\end{remark}

\noindent
Now, we establish the uniform convergence of the proposed operators.
\begin{theorem}
Let $g \in C(\mathcal{D})$ satisfy $g(v)>0$ for all $v \in \mathcal{D}$, and let $f \in C(\mathcal{D})$. Then
\[ \lim_{m\to\infty} \|T_m^{(g)}(f) - f\|_\infty = 0. \]
\end{theorem}

\begin{proof}
Let $f \in C(\mathcal{D})$ and define
\[ f_g(v) := \frac{f(v)}{g(v)}, \quad v \in \mathcal{D}. \] Since $g$ is continuous and strictly positive on the compact interval $\mathcal{D}$, there exist constants $m_g, M_g>0$ such that
\[ 0 < m_g \le g(v) \le M_g, \quad v\in\mathcal{D}. \]
Hence, $f_g \in C(\mathcal{D})$.

\noindent
Using the representation \eqref{repr-gNNO}, we obtain
\begin{align*}
|T_m^{(g)}(f;v) - f(v)|
& = \left| g(v)\,F_m^{(a)}(f_g;v) - g(v)\,f_g(v) \right| \\
& = g(v)\,\left| F_m^{(a)}(f_g;v) - f_g(v) \right| \\
&\le M_g\,\left| F_m^{(a)}(f_g;v) - f_g(v) \right|.
\end{align*}
Taking the supremum over $v\in\mathcal{D}$, we obtain
\begin{equation}\label{eq-final-sup}
\|T_m^{(g)}(f) - f\|_\infty
\le M_g\,\|F_m^{(a)}(f_g) - f_g\|_\infty.
\end{equation}
Thus, it suffices to prove that
\[ \|F_m^{(a)}(f_g) - f_g\|_\infty \to 0 \quad \text{as } m\to\infty. \]
Fix $\varepsilon > 0$. Since $f_g$ is continuous on the compact set $\mathcal{D}$, it is uniformly continuous. Hence, there exists $\delta>0$ such that
\[ |u - v| < \delta \;\Rightarrow\; |f_g(u) - f_g(v)| < \varepsilon,
\quad u,v \in \mathcal{D}.\]
For $v \in \mathcal{D}$, using the weight representation of $F_n^{(a)}$, we write
\[F_m^{(a)}(f_g;v) - f_g(v)
= \sum_{j=\lceil mc\rceil}^{\lfloor md\rfloor} w_{m,j}(v)\big[f_g(j/m) - f_g(v)\big],
\]
where $w_{m,j}(v)\ge 0$ and $\sum_{j=\lceil mc\rceil}^{\lfloor md\rfloor} w_{m,j}(v)=1$.

\noindent
Now, we split the sum into two parts:
\[ \sum_{|j/m - v| < \delta} + \sum_{|j/m - v| \ge \delta}.\]
Then,
\begin{align*}
|F_m^{(a)}(f_g;v)-f_g(v)|
&\le \sum_{|j/m-v|<\delta} w_{m,j}(v)\,|f_g(j/m)-f_g(v)|\\
& \qquad+ \sum_{|j/m-v|\ge \delta} w_{m,j}(v)\,|f_g(j/m)-f_g(v)| \\
&\le \varepsilon + 2\|f_g\|_\infty \sum_{|j/m-v|\ge \delta} w_{m,j}(v).
\end{align*}
Since $a_m \to \mathrm{Id}$ uniformly on $\mathcal{D}$, there exists $N_1$ such that for all $m\ge N_1$,
\[ \sup_{v\in\mathcal{D}} |a_m(v)-v| < \frac{\delta}{2}. \]
Hence, for $m\ge N_1$,
\[ |j/m - v| \ge \delta \;\Rightarrow\; |j/m - a_m(v)| \ge \frac{\delta}{2},\]
and therefore
\[ |j - m a_m(v)| \ge \frac{m\delta}{2}. \]
Thus,
\[ \sum_{|j/m - v| \ge \delta} w_{m,j}(v)
\le \frac{\sum\limits_{|j - m a_m(v)| \ge \frac{m\delta}{2}} \varrho_\rho(m a_m(v)-j)} {\sum\limits_{k=\lceil mc\rceil}^{\lfloor md\rfloor} \varrho_\rho(m a_m(v)-k)}.\]
By property (iv) of Lemma~\ref{lemma-psi-properties}, the numerator tends to zero uniformly with respect to $v\in\mathcal{D}$ as $m\to\infty$. Moreover, by property (v), the denominator is uniformly bounded below by a positive constant $c_0>0$. Hence, there exists $N_2$ such that for all $m\ge N_2$,
\[ \sup_{v\in\mathcal{D}} \sum_{|j/m - v| \ge \delta} w_{m,j}(v) < \frac{\varepsilon}{2\|f_g\|_\infty + 1}.\]
Therefore, for $m \ge \max\{N_1,N_2\}$,
\[\sup_{v\in\mathcal{D}} |F_m^{(a)}(f_g;v)-f_g(v)| \le \varepsilon + \varepsilon = 2\varepsilon.\]
Since $\varepsilon>0$ is arbitrary, we conclude that
\[ \|F_m^{(a)}(f_g) - f_g\|_\infty \to 0. \] Finally, using \eqref{eq-final-sup}, we obtain
\[ \|T_m^{(g)}(f) - f\|_\infty \to 0, \]
which completes the proof.
\end{proof}

\begin{corollary}
Let $\mu>0$ and assume that \( 1+\mu+v > 1, \ v\in\mathcal{D}, \) so that $\ln_\mu(v)=\ln(1+\mu+v)$ is strictly positive on $\mathcal{D}$. Then
\[ \lim_{m\to\infty} \|T_m^{(\log)}(f) - f\|_\infty = 0,
\] for all $f \in C(\mathcal{D})$.
\end{corollary}

\begin{corollary}
Let $\lambda>0$ and define $g(v)=e^{\lambda v}$ for $v\in\mathcal{D}$. Then
\[ \lim_{m\to\infty} \|T_m^{(\exp)}(f) - f\|_\infty = 0,
\] for all $f \in C(\mathcal{D})$.
\end{corollary}

\medskip
Now, we establish quantitative error estimates in terms of the modulus of continuity.

\noindent
For $f\in C(\mathcal{D})$, the \emph{modulus of continuity} is defined by
\[ \omega(f,\delta) := \sup_{\substack{v_1,v_2\in\mathcal{D}\\ |v_1-v_2|\le \delta}} |f(v_1)-f(v_2)|, \quad \delta>0.\]

\begin{theorem}
Let $g \in C(\mathcal{D})$ satisfy $g(v)>0$ for all $v \in \mathcal{D}$, and let $f \in C(\mathcal{D})$. Then, for all $v\in\mathcal{D}$ and $m\in\mathbb{N}$,
\[ |T_m^{(g)}(f;v) - f(v)| \le M_g \left(\frac{\mathcal{M}_1(\varrho_\rho)}{c_0} + C_a + 1\right)\, \omega\!\left(f_g, \tfrac{1}{m}\right),\]
where $f_g(v) := \dfrac{f(v)}{g(v)}$.
\end{theorem}

\begin{proof}
Let $f \in C(\mathcal{D})$ and define
\[c f_g(v) := \frac{f(v)}{g(v)}, \quad v \in \mathcal{D}.\]
Since $g$ is continuous and strictly positive on the compact interval $\mathcal{D}$, there exist constants $m_g, M_g>0$ such that
\[ 0 < m_g \le g(v) \le M_g, \quad v \in \mathcal{D}.\]
Hence, $f_g \in C(\mathcal{D})$.

\noindent
Using the representation \eqref{repr-gNNO}, we write
\begin{align}
|T_m^{(g)}(f;v) - f(v)|
& = g(v)\,\big|F_m^{(a)}(f_g;v) - f_g(v)\big| \nonumber \\
&\le M_g\,\big|F_m^{(a)}(f_g;v) - f_g(v)\big|. 
\label{eq-reduction-final}
\end{align}
Using the weight representation of $F_m^{(a)}$, we have
\[ F_m^{(a)}(f_g;v) - f_g(v) = \sum_{j=\lceil mc\rceil}^{\lfloor md\rfloor}
w_{m,j}(v)\, \big[f_g(j/m) - f_g(v)\big].\]
By the definition of the modulus of continuity,
\[ \left|f_g\!\left(\tfrac{j}{m}\right) - f_g(v)\right| \le \omega\!\left(f_g, \left|\tfrac{j}{m} - v\right|\right).\]
Now, using the standard inequality for the modulus of continuity: for all $s,\delta>0$,
\[ \omega(f,s) \le \left(1 + \frac{s}{\delta}\right)\omega(f,\delta). \]
Taking $\delta = \tfrac{1}{m}$, we obtain
\[
\omega\!\left(f_g, \left|\tfrac{j}{m} - v\right|\right) \le \left(1 + m\left|\tfrac{j}{m} - v\right|\right) \omega\!\left(f_g, \tfrac{1}{m}\right).\]
Therefore,
\begin{align*}
|F_m^{(a)}(f_g;v) - f_g(v)|
&\le \omega\!\left(f_g, \tfrac{1}{m}\right) \left( 1 + m \sum_{j=\lceil mc\rceil}^{\lfloor md\rfloor} w_{m,j}(v)\left|\tfrac{j}{m} - v\right|
\right).
\end{align*}
We now estimate the first absolute moment. Using
\[ \left|\tfrac{j}{m} - v\right| \le \left|\tfrac{j}{m} - a_m(v)\right| + |a_m(v) - v|,\]
we obtain
\[ m \sum_{j=\lceil mc\rceil}^{\lfloor md\rfloor} w_{m,j}(v)\left|\tfrac{j}{m} - v\right| \le I_1 + I_2,\]
where
\[ I_1 := \frac{\sum\limits_{j=\lceil mc\rceil}^{\lfloor md\rfloor}
\varrho_\rho(m a_m(v)-j)\,|j - m a_m(v)|} {\sum\limits_{k=\lceil mc\rceil}^{\lfloor md\rfloor} \varrho_\rho(m a_m(v)-k)}, \quad I_2 := m|a_m(v)-v|.
\]
By the definition of $\mathcal{M}_1(\varrho_\rho)$, we obtain
\[ \sum_{j=\lceil mc\rceil}^{\lfloor md\rfloor} \varrho_\rho(m a_m(v)-j)\,|j - m a_m(v)| \le \mathcal{M}_1(\varrho_\rho).\]
Moreover, by the property (v) of Lemma~\ref{lemma-psi-properties}, the denominator is bounded below by $c_0>0$. Hence,
\[ I_1 \le \frac{\mathcal{M}_1(\varrho_\rho)}{c_0}. \]
By assumption $m\|a_m - \mathrm{Id}\|_\infty \le C_a$, we have
\[ I_2 = m|a_m(v)-v| \le C_a. \] Thus,
\[m \sum_{j=\lceil mc\rceil}^{\lfloor md\rfloor} w_{m,j}(v)\left|\tfrac{j}{m} - v\right|
\le \frac{\mathcal{M}_1(\varrho_\rho)}{c_0} + C_a.\]
Therefore,
\[ |F_m^{(a)}(f_g;v) - f_g(v)| \le \left(\frac{\mathcal{M}_1(\varrho_\rho)}{c_0} + C_a + 1\right) \omega\!\left(f_g, \tfrac{1}{m}\right).\]
Combining with \eqref{eq-reduction-final}, we obtain
\[ |T_m^{(g)}(f;v) - f(v)| \le M_g \left(\frac{\mathcal{M}_1(\varrho_\rho)}{c_0} + C_a + 1\right) \omega\!\left(f_g, \tfrac{1}{m}\right),\]
which completes the proof.
\end{proof}

\begin{corollary}
Let $\mu>0$ and assume that \( 1+\mu+v > 1, \ v\in\mathcal{D},
\) so that $\ln_\mu(v)=\ln(1+\mu+v)$ is strictly positive on $\mathcal{D}$. Define
\[ f_\mu(v) := \frac{f(v)}{\ln_\mu(v)}. \] Then
\[ \left|T_m^{(\log)}(f;v) - f(v)\right| \le M_{\log} \left(\frac{\mathcal{M}_1(\varrho_\rho)}{c_0} + C_a + 1\right) \omega\!\left(f_\mu, \tfrac{1}{m}\right), \ v \in \mathcal{D}.\]
\end{corollary}

\begin{corollary}
Let $\lambda>0$ and define \( f_\lambda(v) := \frac{f(v)}{e^{\lambda v}}, \ v \in \mathcal{D}. \) Then
\[ \left|T_m^{(\exp)}(f;v) - f(v)\right| \le
M_{\exp} \left(\frac{\mathcal{M}_1(\varrho_\rho)}{c_0} + C_a + 1\right)
\omega\!\left(f_\lambda, \tfrac{1}{m}\right), \ v \in \mathcal{D}.
\]
\end{corollary}

\section{Numerical Experiments}
In this section, we present numerical experiments to illustrate the approximation behaviour of the proposed $g$-preserving NNOs and to compare them with the existing operator $F_m$. All computations are carried out in MATLAB using the following procedure.

\begin{algorithm}[H]
\caption{Implementation of the $g$-Preserving NNOs $T_m^{(g)}$}
\begin{algorithmic}[1]
\Require Function $f:[c,d]\to\mathbb{R}$, strictly positive function $g:[c,d]\to(0,\infty)$, activation function $\varrho_\rho$, deformation mapping $a_m$, point $v\in[c,d]$.
\Ensure Approximation $T_m^{(g)}(f;v)$

\State Define index set $j = \lceil mc \rceil, \dots, \lfloor md \rfloor$
\State Compute $u = a_m(v)$
\State Compute weights $w_j = \varrho_\rho(mu - j)$
\State Compute $W = \sum\limits_j w_j$

\If{$W = 0$}
    \State \textbf{error:} normalization failure
\EndIf
\State Set $\tilde{w}_j = w_j / W$
\State Evaluate $f_j = f(j/m)$ and $g_j = g(j/m)$

\If{any $g_j \le 0$}
    \State \textbf{error:} positivity of $g$ violated
\EndIf
\State Compute
\[
T_m^{(g)}(f;v) = g(v)\sum_j \tilde{w}_j \frac{f_j}{g_j}
\]
\State \Return $T_m^{(g)}(f;v)$.
\end{algorithmic}
\end{algorithm}

\subsection{Test and Structure Functions}
Consider the function
\[ f(v)=e^{\lambda v}+\ln(1+\mu+v), \quad v\in[0,1], \] which is smooth and exhibits exponential and logarithmic growth.

\noindent
To investigate the influence of the structure function, we consider the following choices:
\[ g_{\log}(v)=\ln(1+\mu+v), \quad g_{\exp}(v)=e^{\lambda v}, \] with parameters $\mu=1.5$ and $\lambda=0.5$. Both functions are strictly positive and continuous on $[0,1]$, ensuring that the assumptions of the theoretical framework are satisfied.

The activation function is chosen as the logistic sigmoid
\[\rho(v)=\frac{1}{1+e^{-v}}, \quad \varrho_\rho(v)=\frac{1}{2}\big(\rho(v+1)-\rho(v-1)\big).\]
The deformation function is defined by
\[ a_m(v)=v+\frac{1}{m}\,v(1-v), \quad v\in[0,1], \] which represents a small, boundary-preserving perturbation of the identity. In particular, it satisfies
\[ a_m([0,1])\subseteq[0,1], \] and
\[ \|a_m-\mathrm{Id}\|_\infty = \sup_{v\in[0,1]} \left|\frac{1}{m}v(1-v)\right|
= \frac{1}{4m} \to 0 \quad \text{as } m\to\infty.\]
Moreover,
\[ m\|a_m-\mathrm{Id}\|_\infty \le \frac{1}{4}, \] which is consistent with the theoretical assumptions.

We compare the existing operator $F_n$ with the proposed function-preserving operators
\[ T_m^{(\log)} \quad \text{and} \quad T_m^{(\exp)}, \] corresponding to the structure functions $g_{\log}$ and $g_{\exp}$, respectively.

\subsection{Graphical Comparison}
Figure~\ref{fig:approximation} depicts how the approximation of the function $f$ using the existing operator and the function-preserving operators for a representative value of $m$.

\begin{figure}[H]
\centering
\includegraphics[width=0.8\textwidth]{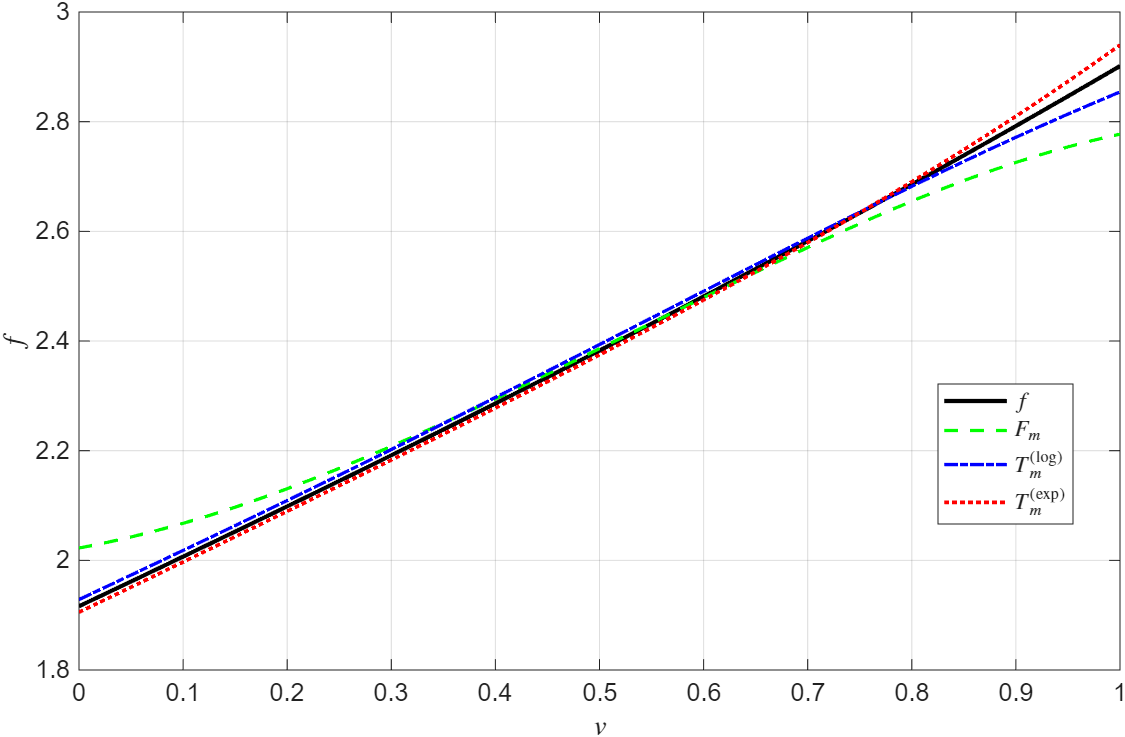}
\caption{Approximation of $f$ using $F_m$, $T_m^{(\log)}$ and $T_m^{(\exp)}$ for $m=10$.}
\label{fig:approximation}
\end{figure}

All operators provide accurate approximations of the target function. Depending on the selected function $g$, the function-preserving operators adjust slightly, resulting in minor differences in the approximation quality, especially for moderate values of $m$.

 In order to verify the reproduction property, 
\[ T_m^{(g)}(g;v)=g(v). \] We use the corresponding operator to approximate each structure function.

\begin{figure}[H]
\centering
\includegraphics[width=1\textwidth]{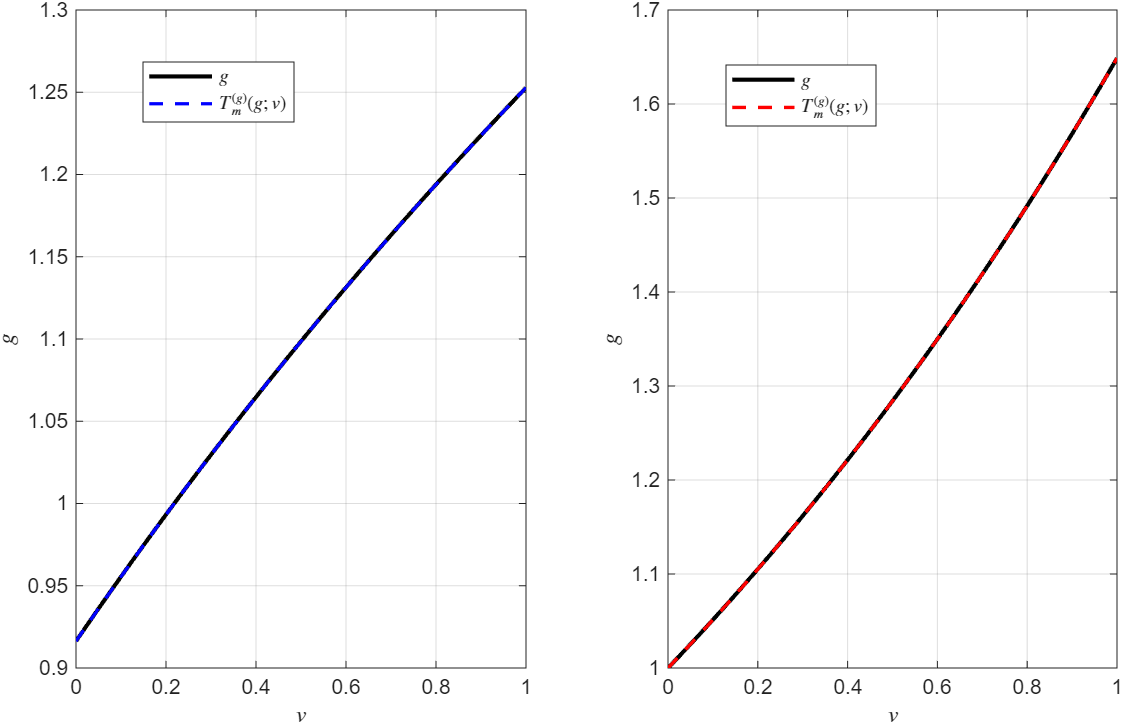}
\caption{Verification of the function-preserving property for $g_{\log}$ and $g_{\exp}$.}
\label{fig:preservation}
\end{figure}

\begin{table}[th]
\centering
\begin{tabular}{|lc|}
\hline
Quantity & Value \\
\hline
Reproduction error ($\log$) & $5.5511 \times 10^{-16}$ \\
Reproduction error ($\exp$) & $6.6613 \times 10^{-16}$ \\
\hline
\end{tabular}
\caption{Observed reproduction errors for the proposed operators.}
\label{tab:repro_errors}
\end{table}

The numerical results in Figure~\ref{fig:preservation} and Table~\ref{tab:repro_errors} confirm that the operators reproduce the corresponding structure functions with very high accuracy.

\subsection{Error Analysis}
To quantify the approximation performance, we compute the supremum errors
\[E_m = \|F_m(f)-f\|_\infty, \quad E_m^{(\log)} = \|T_m^{(\log)}(f)-f\|_\infty, \quad E_m^{(\exp)} = \|T_m^{(\exp)}(f)-f\|_\infty.\]
The error decay on the log-log scale is presented in Figure~\ref{fig:error-log}.

\begin{figure}[H]
\centering
\includegraphics[width=0.8\textwidth]{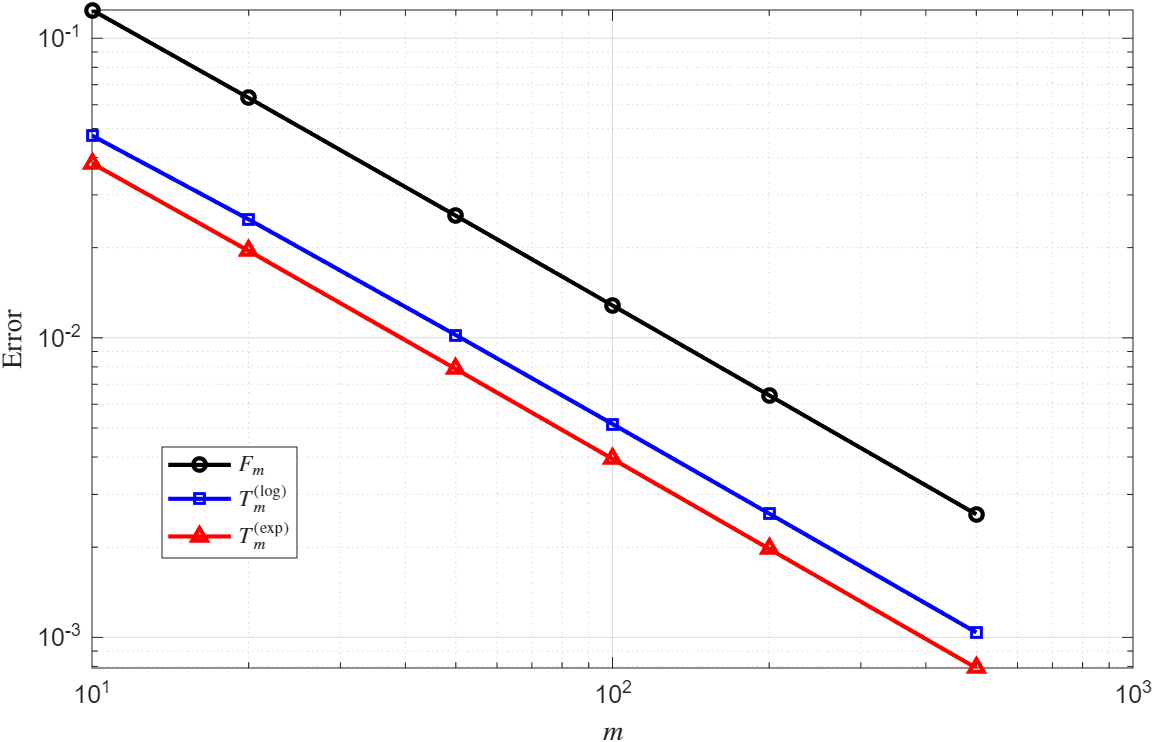}
\caption{Log-log plot of the approximation errors vs. $m$.}
\label{fig:error-log}
\end{figure}

\noindent
The results indicate that all operators exhibit similar convergence rates. The slopes of the log-log plots are close to $-1$, which is consistent with the expected behaviour for continuous Lipschitz functions.

\noindent
For completeness, the error behaviour on a normal scale is shown in Figure~\ref{fig:error-normal} and Table~\ref{tab:sup_errors}.

\begin{table}[th]
\centering
\begin{tabular}{|l|c|c|c|l|}
\hline
$m$ & $E_m$  & $E_m^{(\log)}$  & $E_m^{(\exp)}$  \\
\hline
10  & $1.242902\times 10^{-1}$ & $4.735552\times 10^{-2}$ & $3.827629\times 10^{-2}$ \\
20  & $6.323344\times 10^{-2}$ & $2.485107\times 10^{-2}$ & $1.951285\times 10^{-2}$ \\
50  & $2.556789\times 10^{-2}$ & $1.021837\times 10^{-2}$ & $7.870158\times 10^{-3}$ \\
100 & $1.283134\times 10^{-2}$ & $5.155500\times 10^{-3}$ & $3.944216\times 10^{-3}$ \\
200 & $6.427675\times 10^{-3}$ & $2.589334\times 10^{-3}$ & $1.974244\times 10^{-3}$ \\
500 & $2.573971\times 10^{-3}$ & $1.038515\times 10^{-3}$ & $7.901916\times 10^{-4}$ \\
\hline
\end{tabular}
\caption{Supremum uniform errors for the operators $F_m$, $T_m^{(\log)}$, and $T_m^{(\exp)}$.}
\label{tab:sup_errors}
\end{table}

\begin{figure}[H]
\centering
\includegraphics[width=0.8\textwidth]{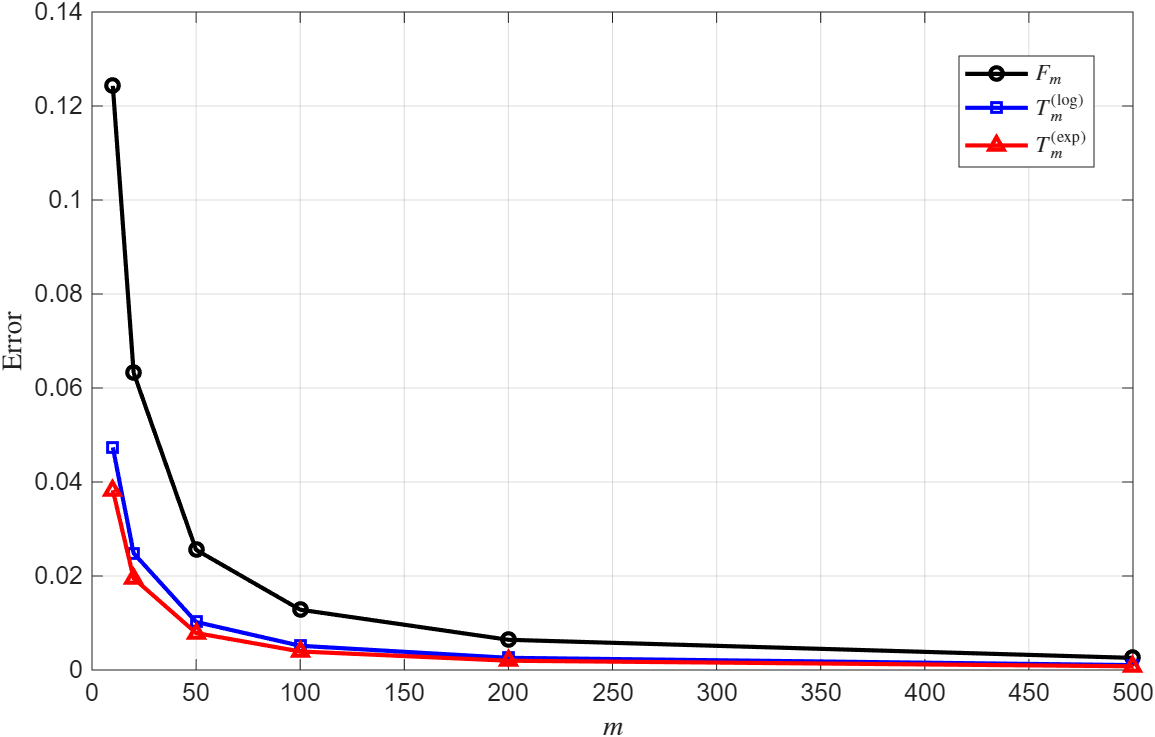}
\caption{Supremum error vs. $m$.}
\label{fig:error-normal}
\end{figure}

\noindent
It is observed that function-preserving operators can achieve slightly improved accuracy compared to the existing operator, depending on how well the chosen structure function reflects the behaviour of $f$.

\subsection{Discussion}
The experiments demonstrate that all operators achieve the same order of approximation, while the choice of structure function $g$ influences the accuracy for finite $m$. In particular, improved performance is observed when the structure function captures the dominant features of the target function.

\noindent
This behaviour is consistent with the theoretical estimate
\[ |T_m^{(g)}(f;v)-f(v)| \le C\,\omega\!\left(\frac{f}{g},\frac{1}{m}\right), \]
which shows that the approximation error depends on the smoothness of the transformed function $f/g$.

Overall, the numerical results show that the $g$-preserving NNOs give a flexible and efficient framework for function-preserving approximation and validate the theoretical conclusions.

\section{Application: Denoising of Modulated Signals}
In this section, the efficiency of the proposed NNOs for denoising oscillatory signals with a multiplicative structure is demonstrated. In particular, we consider signals of the kind
\[ f(v) = g(v)\,h(v), \quad v \in [0,2], \] where $h(v)$ refers to a high-frequency carrier component and $g(v)$ represents a slowly varying envelope.

\subsection{Test Signal Construction}
We take 
\(g(v) = \sin(2\pi v), \ h(v) = \cos(40\pi v),\) so that
\[f(v) = \sin(2\pi v)\cos(40\pi v).\] The aforementioned signal is an appropriate test case for approximation under noise, as it exhibits fast oscillations modulated by a smooth envelope.

A noisy signal is generated as
\[f_{\text{noisy}}(v) = f(v) + \varepsilon(v),\] where $\varepsilon(v)$ is Gaussian noise with zero mean and variance $\sigma^2 = 0.15^2$. 

Numerical instability may result from a direct ratio-based approximation since the envelope $g(v)$ vanishes at certain points. We address this by introducing a modified envelope that is strictly positive:
\[g_{\text{safe}}(v) = | \sin(2\pi v) | + 0.2,\]
which assures
\[g_{\text{safe}}(v) \geq 0.2, \quad \forall v \in [0,2].\]
\subsection{Operator Setup}
The NNOs, as defined in \ref{def-gNNO} are taken into consideration
\[T_m^{(g)}(f;v) = g_{\text{safe}}(v) \sum_{j=\lceil mc \rceil}^{\lfloor md \rfloor} \omega_{j,m}(v)\, \frac{f_{\text{noisy}}(j/m)}{g_{\text{safe}}(j/m)},\]
where the normalized weights are given by
\[\omega_{j,m}(v) = \frac{\varrho_\rho\big(m a_m(v) - j\big)}
{\sum_{k=\lceil mc \rceil}^{\lfloor md \rfloor} \varrho_\rho\big(m a_m(v) - k\big)}.\]
Here, the activation function is defined via a $\tanh$ sigmoidal function as
\[\rho(v) = \frac{1 + \tanh(v)}{2}, \qquad \varrho_\rho(v) = \frac{1}{2}\big(\rho(v+1) - \rho(v-1)\big),
\] and deformation mapping is given by \[a_m(v) = v + \frac{1}{m}v(1-v).\]
\subsection{Numerical Results}
The performance of the operator $T_m^{(g)}$ for $m=150$ is shown in Figure~\ref{fig:denoising}, where the noisy signal, the envelope, and the reconstructed signal are compared.

\begin{figure}[H]
\centering
\includegraphics[width=0.9\textwidth]{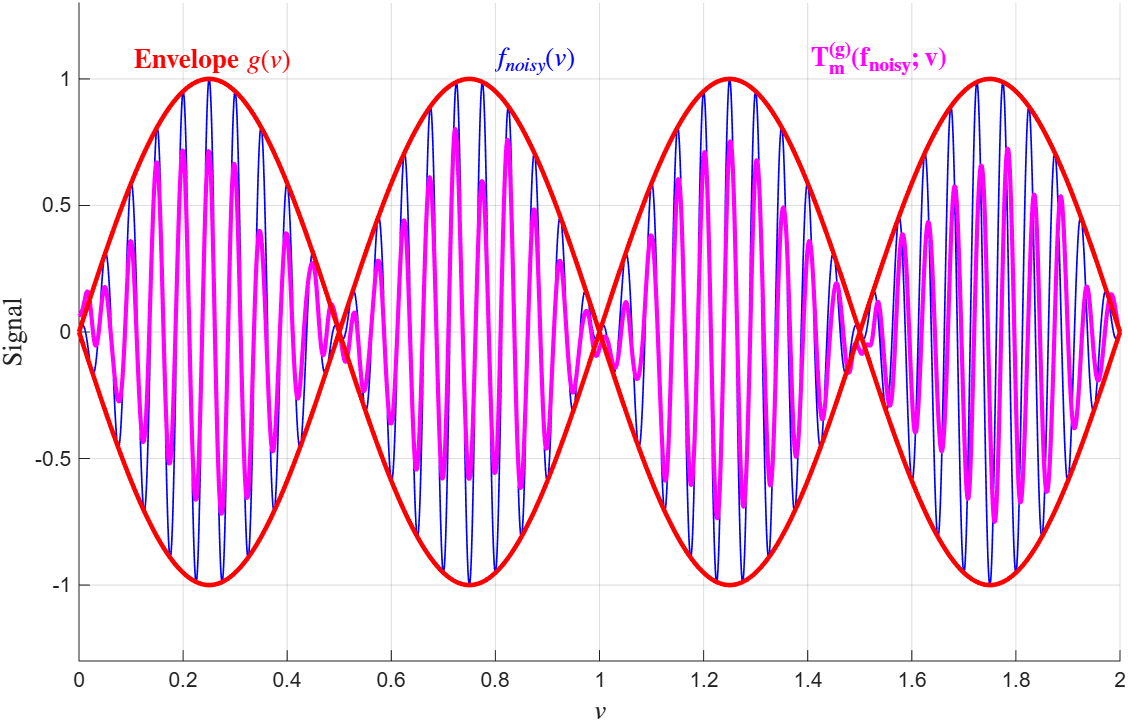}
\caption{A modulated signal is denoised using the operator $T_m^{(g)}$. The envelope (red), noisy signal (blue), and reconstructed signal (magenta) are presented.}
\label{fig:denoising}
\end{figure}

\subsection{Discussion}
The numerical results show that the proposed operator effectively reduces noise while preserving the signal's underlying harmonic structure. The main benefits of this strategy are:

\begin{itemize}
\item \textbf{Structure-aware approximation:} The approach improves accuracy by approximating the smoother component $f/g$ rather than directly calculating $f$.

\item \textbf{Stability:} Near the envelope's zeros, numerical instability is avoided by introducing $g_{\text{safe}}$.

\item \textbf{Noise robustness:} The averaging nature of the operator reduces the impact of additive noise.
\end{itemize}

The potential of proposed NNOs in signal processing applications is demonstrated by this example, especially for structured signals with multiplicative components.

\section{Conclusion}
In this work, we develop a novel class of \emph{function-preserving neural network operators (NNOs)} by adding a strictly positive function $g$ into existing NNOs constructions, and the approximation method is able to more accurately reflect the target function's intrinsic structure.

It is demonstrated that the proposed operators are linear, positive, well-defined, and able to precisely mimic the function $g$. A detailed analysis of their approximation behaviour establishes uniform convergence on compact intervals and derives error estimates in terms of the modulus of continuity.

Furthermore, numerical studies corroborate the theoretical results by showing that, particularly for moderate values of the discretization parameter, an appropriate choice of the function $g$ increases approximation accuracy without impacting the convergence rate. In addition, the practical applicability of the proposed operators is demonstrated through their use in modulated signal denoising, where the underlying multiplicative structure of the signal is effectively exploited to achieve stable and accurate reconstruction in the presence of noise.

Naturally, the operator described here provides several opportunities for further investigation. The development of Kantorovich-type and higher-order operator variations is one important extension of these. Further understanding of the approximation behaviour and the identification of optimal function classes would be made possible by the investigation of asymptotic results, such as saturation events and Voronovskaja-type theorems.

\bigskip
\noindent
\textbf{Declaration of competing interest}

The author affirms that the work presented in this paper is not influenced by any conflicting financial interests or personal connections.

\noindent\textbf{Author contributions} 

The author contributed to all aspects of the work, including its conception, analysis, and writing.

\noindent\textbf{Data availability}

No datasets were generated or analyzed during the current study.

\noindent\textbf{Funding}

The author declares that no funds, grants, or other financial support were received for the research, authorship, or publication of this work.

\bibliographystyle{elsarticle-num}
\bibliography{Ref}

@article{cardaliaguet1992approximation,
  title={Approximation of a function and its derivative with a neural network},
  author={Cardaliaguet, Pierre and Euvrard, Guillaume},
  journal={Neural Netw.},
  volume={5},
  number={2},
  pages={207--220},
  year={1992},
  publisher={Elsevier}
}

@article{anastassiou1997rate,
  title={Rate of convergence of some neural network operators to the unit-univariate case},
  author={Anastassiou, George A},
  journal={J. Math. Anal. Appl.},
  volume={212},
  number={1},
  pages={237--262},
  year={1997},
  publisher={Elsevier}
}

@article{anastassiou2000rate,
  title={Rate of convergence of some multivariate neural network operators to the unit},
  author={Anastassiou, George A},
  journal={Comput. Math. Appl.},
  volume={40},
  number={1},
  pages={1--19},
  year={2000},
  publisher={Elsevier}
}

@article{anastassiou2011multivariatesigmoidal,
  title={Multivariate sigmoidal neural network approximation},
  author={Anastassiou, George A},
  journal={Neural Netw.},
  volume={24},
  number={4},
  pages={378--386},
  year={2011},
  publisher={Elsevier}
}

@article{anastassiou2013frac,
  title={Fractional approximation by normalized bell and squashing type neural network operators},
  author={Anastassiou, George A},
  journal={New Math. Nat. Comput.	},
  volume={9},
  number={01},
  pages={43--63},
  year={2013},
  publisher={World Scientific}
}

@article{anastassiou2017approximations,
  title={Approximations by multivariate perturbed neural network operators},
  author={Anastassiou, George A},
  journal={Analysis and Applications},
  volume={15},
  number={03},
  pages={413--432},
  year={2017},
  publisher={World Scientific}
}

@article{costarelli2013approxi,
  title={Approximation results for neural network operators activated by sigmoidal functions},
  author={Costarelli, Danilo and Spigler, Renato},
  journal={Neural Netw.},
  volume={44},
  pages={101--106},
  year={2013},
  publisher={Elsevier}
}

@article{costarelli2014convergence,
  title={Convergence of a family of neural network operators of the {K}antorovich type},
  author={Costarelli, Danilo and Spigler, Renato},
  journal={J. Approx. Theory},
  volume={185},
  pages={80--90},
  year={2014},
  publisher={Elsevier}
}

@article{costarelli2022density,
  title={Density results by deep neural network operators with integer weights},
  author={Costarelli, Danilo},
  journal={Math. Model. Anal.},
  volume={27},
  number={4},
  pages={547--560},
  year={2022}
}

@article{costarelli2024asymptotic,
  title={Asymptotic analysis of neural network operators employing the {H}ardy-{L}ittlewood maximal inequality},
  author={Costarelli, Danilo and Piconi, Michele},
  journal={Mediterr. J. Math.},
  volume={21},
  number={7},
  pages={199},
  year={2024},
  publisher={Springer}
}

@article{baxhaku2025riemann,
  title={Riemann-{L}iouville Fractional Integral Type Deep Neural Network Kantorovich Operators},
  author={Baxhaku, Behar and Agrawal, Purshottam Narain and Bajpeyi, Shivam},
  journal={Iran. J. Sci.},
  volume={49},
  number={3},
  pages={711--724},
  year={2025},
  publisher={Springer}
}

@article{bajpeyi2021exponential,
  title={Approximation by exponential sampling type neural network operators},
  author={Bajpeyi, Shivam and Kumar, A. Sathish},
  journal={Anal. Math. Phys.},
  year={2021}
}

@article{bajpeyi2023order,
  title={Order of approximation for exponential sampling type neural network operators},
  author={Bajpeyi, Shivam},
  journal={Results Math.},
  volume={78},
  number={3},
  pages={99},
  year={2023},
  publisher={Springer}
}

@article{majethiya2025modular,
  title={Modular convergence of semi-discrete neural network operators in Orlicz space: P. Majethiya, S. Bajpeyi},
  author={Majethiya, Priyanka and Bajpeyi, Shivam},
  journal = {J. Pseudo-Differ. Oper. Appl.},
  volume={16},
  number={2},
  pages={27},
  year={2025},
  publisher={Springer}
}

@book{lorentz1986bernstein,
  title={Bernstein Polynomials},
  author={Lorentz, G.G.},
  year={1986},
  publisher={Chelsea Publishing Company}
}

@article{cardenas2006shape,
  title={Shape preserving approximation by {B}ernstein-type operators which fix polynomials},
  author={C{\'a}rdenas-Morales, D and Garrancho, P and Mu{\~n}oz-Delgado, Francisco-Javier},
  journal={Appl. Math. Comput.},
  volume={182},
  number={2},
  pages={1615--1622},
  year={2006},
  publisher={Elsevier}
}

@article{bede2009approximation,
  title={Approximation and Shape Preserving Properties of the {B}ernstein Operator of Max-Product Kind},
  author={Bede, Barnabas and Coroianu, Lucian and Gal, Sorin G},
  journal={Int. J. Math. Math. Sci.},
  volume={2009},
  number={1},
  pages={590589},
  year={2009},
  publisher={Wiley Online Library}
}

@article{aldaz2009shape,
  title={Shape preserving properties of generalized {B}ernstein operators on extended Chebyshev spaces},
  author={Aldaz, JM and Kounchev, Ognyan and Render, Hermann},
  journal={Numer. Math.},
  volume={114},
  number={1},
  pages={1--25},
  year={2009},
  publisher={Springer}
}

@article{acu2024new,
  title={A new representation and shape-preserving properties of perturbed {B}ernstein operators},
  author={Acu, Ana-Maria and Mutlu, G{\"o}khan and {\c{C}}ekim, Bayram and Yaz{\i}c{\i}, Serdal},
  journal={Math. Methods Appl. Sci.},
  volume={47},
  number={1},
  pages={5--14},
  year={2024},
  publisher={Wiley Online Library}
}

@article{su2022shape,
  title={On the shape-preserving properties of $\lambda$-{B}ernstein operators},
  author={Su, Lian-Ta and Mutlu, G{\"o}khan and {\c{C}}ekim, Bayram},
  journal={J. Inequal. Appl.},
  volume={2022},
  number={1},
  pages={151},
  year={2022},
  publisher={Springer}
}

@article{ayman2025shape,
  title={Shape preserving approximation properties to family of $\alpha$-{B}ernstein shifted knot operators},
  author={Ayman-Mursaleen, Mohammad and Alshaban, Esmail and Nasiruzzaman, Md},
  journal={J. Inequal. Appl.},
  volume={2025},
  number={1},
  pages={107},
  year={2025},
  publisher={Springer}
}

@article{berisha2025riemann,
  title={Riemann-{L}iouville fractional exponential sampling type neural network {K}antorovich operators},
  author={Berisha, Artan and Baxhaku, Fesal and Agrawal, Purshottam Narain and Baxhaku, Behar},
  journal={Eng. Appl. Artif. Intell.},
  volume={151},
  pages={110477},
  year={2025},
  publisher={Elsevier}
}

@article{saini2026approximation,
  title={Approximation capabilities of certain nonlinear neural network operators},
  author={Saini, Sachin and Singh, Uaday},
  journal={J. Comput. Appl. Math.},
  pages={117452},
  year={2026},
  publisher={Elsevier}
}

@article{costarelli2022quantitative,
  title={Quantitative estimates for neural network operators implied by the asymptotic behaviour of the sigmoidal activation functions},
  author={Coroianu, Lucian and Costarelli, Danilo and Kadak, U{\u{g}}ur},
  journal={Mediterr. J. Math.},
  volume={19},
  number={5},
  pages={211},
  year={2022},
  publisher={Springer}
}

@article{costarelli2024durrmeyer,
  title={The approximation capabilities of Durrmeyer-type neural network operators},
  author={Coroianu, Lucian and Costarelli, Danilo and Natale, Mariarosaria and Panti{\c{s}}, Alexandra},
  journal={J. Appl. Math. Comput.},
  pages={1--19},
  year={2024},
  publisher={Springer}
}

@article{kadak2023neural,
  title={Neural network operators of generalized fractional integrals equipped with a vector-valued function},
  author={Kadak, Ugur and Costarelli, Danilo and Coroianu, Lucian},
  journal={Chaos, Solitons \& Fractals},
  volume={177},
  pages={114272},
  year={2023},
  publisher={Elsevier}
}

@article{lin2026construction,
  title={Construction and approximation rate of neural network operators for Broad Learning System},
  author={Lin, Huijie and Yu, Dansheng and Cao, Feilong},
  journal={Chaos, Solitons \& Fractals},
  volume={207},
  pages={118048},
  year={2026},
  publisher={Elsevier}
}

@article{chen2015approximation,
  title={Approximation by network operators with logistic activation functions},
  author={Chen, Zhixiang and Cao, Feilong and Hu, Jinjie},
  journal={Appl. Math. Comput.},
  volume={256},
  pages={565--571},
  year={2015},
  publisher={Elsevier}
}

@article{qian2022neural,
  title={Neural network interpolation operators activated by smooth ramp functions},
  author={Qian, Yunyou and Yu, Dansheng},
  journal={Analysis and Applications},
  volume={20},
  number={04},
  pages={791--813},
  year={2022},
  publisher={World Scientific}
}

@article{xue2025direct,
  title={Direct and Converse Results of Approximation by {K}antorovich Type Neural Network Operators in Orlicz Spaces: R. Xue and D. Yu},
  author={Xue, Renyi and Yu, Dansheng},
  journal={Mediterr. J. Math.},
  volume={22},
  number={6},
  pages={145},
  year={2025},
  publisher={Springer}
}
\end{document}